\documentclass[letterpaper, 11pt,  reqno]{amsart}

\usepackage{amsmath,amssymb,amscd,amsthm,amsxtra, esint}

\allowdisplaybreaks[2]

\newtheorem{theorem}{Theorem} [section]

\newtheorem{lemma}[theorem]{Lemma}
\newtheorem{proposition}[theorem]{Proposition}
\newtheorem{remark}[theorem]{Remark}

\DeclareMathOperator*{\supp}{supp}

\newcommand{\noi}{\noindent}
\newcommand{\Z}{\mathbb{Z}}
\newcommand{\R}{\mathbb{R}}

\newcommand{\T}{\mathbb{T}}

\let\P= \undefined
\newcommand{\P}{\mathbf{P}}

\newcommand{\N}{\mathcal{N}}

\newcommand{\F}{\mathcal{F}}

\newcommand{\al}{\alpha}

\newcommand{\dl}{\delta}

\newcommand{\nb}{\nabla}

\newcommand{\Dl}{\Delta}
\newcommand{\eps}{\varepsilon}

\newcommand{\g}{\gamma}
\newcommand{\G}{\Gamma}
\newcommand{\ld}{\lambda}

\newcommand{\s}{\sigma}

\newcommand{\ft}{\widehat}

\newcommand{\wt}{\widetilde}

\newcommand{\dt}{\partial_t}

\renewcommand{\l}{\ell}
\renewcommand{\o}{\omega}
\renewcommand{\O}{\Omega}

\newcommand{\les}{\lesssim}
\newcommand{\ges}{\gtrsim}

\newcommand{\jb}[1]
{\langle #1 \rangle}

\numberwithin{equation}{section}
\numberwithin{theorem}{section}

\begin{document}

\baselineskip = 15pt

\title[Randomization on unbounded domains]
{Wiener randomization on unbounded domains
and an application to almost sure well-posedness  of NLS}

\author[\'A.~B\'enyi, T.~Oh, and O.~Pocovnicu]
{\'Arp\'ad  B\'enyi, Tadahiro Oh, and Oana Pocovnicu}

\address{
\'Arp\'ad  B\'enyi\\
Department of Mathematics\\
 Western Washington University\\
 516 High Street, Bellingham\\
  WA 98225\\ USA}
\email{arpad.benyi@wwu.edu}

\address{
Tadahiro Oh\\
School of Mathematics\\
The University of Edinburgh,
and The Maxwell Institute for the Mathematical Sciences\\
James Clerk Maxwell Building\\
The King's Buildings\\
Mayfield Road\\
Edinburgh\\
EH9 3JZ, United Kingdom}

\email{hiro.oh@ed.ac.uk}

\address{
Oana Pocovnicu\\
School of Mathematics\\
Institute for Advanced Study\\
Einstein Drive, Princeton\\ NJ 08540\\ USA}

\email{pocovnicu@math.ias.edu}



\thanks{This work is partially supported by a grant from the Simons Foundation (No.~246024 to \'Arp\'ad B\'enyi).}

\thanks{This material is based upon work supported by the National Science Foundation under agreement No. DMS-1128155. 
Any opinions, findings, and conclusions or recommendations expressed in this material are those of the author and do not necessarily reflect the views of the National Science Foundation.}

\begin{abstract}
We consider a randomization of a function on $\R^d$ that is naturally associated to the Wiener decomposition and, intrinsically, to the modulation spaces. Such randomized functions enjoy better integrability, thus allowing us to improve the Strichartz estimates for the Schr\"odinger equation. As an example, we also show that the energy-critical cubic nonlinear Schr\"odinger equation on $\R^4$ is almost surely locally well-posed with respect to randomized initial data below the energy space.
\end{abstract}

%
\maketitle
%

\baselineskip = 15pt

\section{Introduction}

\subsection{Background}
The Cauchy problem of the nonlinear Schr\"odinger equation (NLS):
\begin{equation}
\begin{cases}\label{NLS1}
i \partial_t u + \Delta u = \pm |u|^{p-1}u \\
u\big|_{t = 0} = u_0 \in H^s(\R^d),
\end{cases}
\qquad ( t, x) \in \R \times \R^d
\end{equation}

\noi
has been studied extensively
over recent years.
One of the key ingredients in studying \eqref{NLS1}
is the dispersive effect of the associated linear flow.
Such dispersion is
often expressed in terms of
the Strichartz estimates
(see Lemma \ref{LEM:Str} below),
which have played an important role in
studying various problems on \eqref{NLS1},
in particular,  local and global well-posedness
issues.

It is  well-known that \eqref{NLS1} is invariant under several symmetries.
In the following, we concentrate on the dilation symmetry.
The dilation symmetry states that if $u(t, x)$ is a solution to \eqref{NLS1}
on $\R^d$ with an initial condition $u_0$, then
$u^\ld(t, x) = \ld^{-\frac{2}{p-1}} u (\ld^{-2}t, \ld^{-1}x)$
is also a solution to \eqref{NLS1} with the $\ld$-scaled initial condition
$u_0^\ld(x) = \ld^{-\frac{2}{p-1}} u_0 (\ld^{-1}x)$.
Associated to the dilation symmetry,
there is a  scaling-critical Sobolev index $s_c := \frac d2 - \frac{ 2}{p-1}$
such that the homogeneous $\dot{H}^{s_c}$-norm is invariant
under the dilation symmetry.
For example, when $p = \frac{4}{d-2}+1$,
we have $s_c = 1$ and \eqref{NLS1} is called energy-critical.
It is known that \eqref{NLS1} is ill-posed
in the supercritical regime, that is, in $H^s$ for $s < s_c$;
see \cite{CCT, BGT2, Carles, AC}.

In an effort to study the invariance of the Gibbs measure
for the defocusing (Wick ordered) cubic NLS on $\T^2$,
Bourgain \cite{BO7} considered random initial data of the form:
\begin{equation}
 u_0^\omega(x) = \sum_{n \in \Z^2} \frac{g_n(\o)}{\sqrt{1+|n|^2}}e^{i n \cdot x},
\label{I1}
 \end{equation}

\noi
where $\{g_n\}_{n \in \Z^2}$ is a sequence of independent
complex-valued standard Gaussian random variables.
The function \eqref{I1} represents a typical
element in the support of the Gibbs measure,
more precisely,  in the support of the Gaussian free field on $\T^2$
associated to this Gibbs measure,
and is critical with respect to the scaling.
With a combination of deterministic PDE techniques
and probabilistic arguments,
Bourgain showed that
 the (Wick ordered) cubic NLS on $\T^2$ is well-posed
almost surely with respect to random initial data \eqref{I1}.
Burq-Tzvetkov \cite{BT2} further explored
the study of Cauchy problems with more general random initial data.
They considered the cubic nonlinear wave equation (NLW)
on a three dimensional compact Riemannian manifold $M$ without a boundary,
where the scaling-critical Sobolev index $s_c$ is given by $s_c = \frac 12 $.
Given $u_0(x) = \sum_{n = 1}^\infty  c_n e_n(x) \in H^s(M)$, $s\geq \frac 14$,
they proved almost sure local well-posedness
with random initial data of the form:\footnote{For NLW, one needs to specify
$(u, \dt u)|_{t = 0}$ as an initial condition. For simplicity of presentation, we
only displayed $u|_{t = 0}$ in \eqref{I2}.}
\begin{equation}
u_0^\omega (x) = \sum_{n = 1}^\infty g_n (\omega) c_n e_n(x)
\label{I2}
\end{equation}
	
\noi
where $\{g_n\}_{n = 1}^\infty$ is a sequence of independent mean-zero
random variables with a uniform bound on the fourth moments
and $\{e_n\}_{n = 1}^\infty$ is an orthonormal basis of $L^2(M)$
consisting of the eigenfunctions of the Laplace-Beltrami operator.
It was also shown that $u_0^\o$ in \eqref{I2}
has the same Sobolev regularity as the original function $u_0$
and is not smoother, almost surely.
In particular, if $u_0 \in H^s(M) \setminus H^\frac{1}{2}(M)$,
their result implies almost sure local well-posedness
in the supercritical regime. There are several works on Cauchy problems of evolution equations with random data that followed these results,
including some on almost sure global well-posedness:
\cite{Bo97, Thomann, CO, BTT,  Deng, DS1, BT3, NPS, DS2, R, BTT2,  BB1, BB2, NS, PRT, LM}.

We point out that many of these works are on compact domains,
where there is  a  countable basis of eigenfunctions of the Laplacian
and thus there is a natural way to introduce a randomization.
On $\R^d$, randomizations were introduced with respect to a countable basis of eigenfunctions of the Laplacian with a confining potential such as a harmonic oscillator $\Dl - |x|^2$; we note that functions in Sobolev spaces associated to the Laplacian with a confining potential have an extra decay in space. Our goal is to introduce a randomization for functions in the usual Sobolev spaces on $\R^d$ without such extra decay.
For this purpose, we first review some basic notions and facts concerning the so-called \emph{modulation spaces} of time-frequency analysis.

\subsection{Modulation spaces}
The modulation spaces were introduced by Feichtinger \cite{Fei}
in early eighties. In following collaborations with Gr\"ochenig \cite{FG1, FG2}, they established the basic theory of these function spaces, in particular their invariance, continuity, embeddings, and convolution properties. The difference between the Besov spaces and the modulation spaces consists in the geometry of the frequency space employed: the dyadic annuli in the definition of the former spaces are replaced by unit cubes $Q_n$ centered at $n\in \Z^d$ in the definition of the latter ones. Thus, the modulation spaces arise via a uniform partition of the frequency space
$\R^d = \bigcup_{n \in \Z^d} Q_n$,
which is commonly referred to as a~{\it Wiener decomposition} \cite{W}.
In certain contexts, this decomposition allows for a finer analysis by effectively capturing the time-frequency
concentration of a distribution.

For $x, \xi\in \R^d$, let $\mathcal Fu (\xi)=\widehat u(\xi)=\int_{\R^d} u(x)e^{-2\pi ix\cdot\xi}\,dx$ denote the Fourier transform of a distribution $u$. Typically, the (weighted) modulation spaces $M^{p, q}_s(\R^d)$, $p, q>0, s\in \R$, are defined by imposing the $L^p(dx)L^q(\jb{\xi}^s d\xi)$ integrability of the short-time (or windowed) Fourier transform of a distribution $V_\phi u(x, \xi) :=\mathcal F(u\overline{T_x\phi})(\xi)$.
Here,
$\langle \xi\rangle^s=(1+|\xi|^2)^{\frac s2}$,
$\phi$ is some fixed non-zero Schwartz function, and $T_x$ denotes the translation
defined by $T_x(\phi)(y) = \phi(y - x)$.
When $s=0$, one simply writes $M^{p,q}$. Modulation spaces satisfy some desirable properties: they are quasi-Banach spaces, two different windows $\phi_1, \phi_2$ yield equivalent norms, $M_s^{2,2}(\R^d)=H^s(\R^d)$,
$(M_s^{p,q}(\R^d))'=M_{-s}^{p',q'}(\R^d)$,
$M_{s_1}^{p_1, q_1}(\R^d)\subset M_{s_2}^{p_2, q_2}(\R^d)$
for $s_1 \geq s_2$, $p_1\leq p_2$,
and $q_1\leq q_2$, and $\mathcal S(\R^d)$ is dense
in $M_s^{p,q}(\R^d)$.

We prefer to use an equivalent norm on the modulation space $M^{p, q}_s$, which is induced by a corresponding Wiener decomposition of the frequency space. Given $\psi \in \mathcal{S}(\R^d)$ such that $\supp \psi \subset [-1,
1]^d$ and $\sum_{n \in \Z^d} \psi(\xi -n) \equiv 1$,
let
\begin{equation}
\|u\|_{M_s^{p, q}(\R^d)} = \big\| \jb{n}^s \|\psi(D-n) u
\|_{L_x^p(\R^d)} \big\|_{\l^q_n(\mathbb{Z}^d)}.
\label{mod2}
\end{equation}

\noi
Note that $\psi(D-n)$ is just a Fourier multiplier with symbol $\chi_{Q_n}$ conveniently smoothed:
\[\psi(D-n)u(x)=\int_{\R^d}\psi(\xi-n)\widehat u(\xi)e^{2\pi ix\cdot \xi}\,d\xi.\]
It is worthwhile to compare the definition \eqref{mod2}
with the one for the Besov spaces which uses a dyadic partition of the frequency domain. Let $\varphi_0, \varphi \in \mathcal{S}(\R^d)$ such that
$\supp \varphi_0 \subset \{ |\xi| \leq 2\}$, $\supp \varphi \subset
\{ \frac{1}{2}\leq |\xi| \leq 2\}$, and $ \varphi_0(\xi) + \sum_{j =
1}^\infty \varphi(2^{-j}\xi) \equiv 1.$ With $\varphi_j(\xi) =
\varphi(2^{-j}\xi)$,
we define  the Besov spaces $B_s^{p, q}$
via the norm
\begin{equation} \label{besov1}
\|u\|_{B_s^{p, q}(\R^d) } = \big\| 2^{js} \|\varphi_j(D) u
\|_{L^p(\R^d)} \big\|_{\l^q_j(\mathbb{Z}_{\geq 0})}.
\end{equation}

\noi
There are several known embeddings between the Besov, Sobolev, and
modulation spaces; see,  for example,  Okoudjou \cite{Ok}, Toft \cite{To}, Sugimoto-Tomita \cite{suto2},  and Kobayashi-Sugimoto \cite{kosu}.

\subsection{Randomization adapted to the Wiener decomposition}
Given a function $\phi$ on $\R^d$,
we have
\[ \phi = \sum_{n \in \Z^d} \psi(D-n) \phi,\]

\noi
where $\psi(D-n)$ is defined above.
 The identity above leads to a randomization that is
naturally associated to the Wiener decomposition, and hence to the modulation spaces, as follows.
Let $\{g_n\}_{n \in \Z^d}$ be a sequence of independent mean zero complex-valued random variables
on a probability space $(\O, \F, P)$,
where the real and imaginary parts of $g_n$ are independent
and endowed
with probability distributions $\mu_n^{(1)}$ and $\mu_n^{(2)}$.
Then, we can define the \emph{Wiener randomization of $\phi$} by
\begin{equation}
\phi^\omega : = \sum_{n \in \Z^d} g_n (\omega) \psi(D-n) \phi.
\label{R1}
\end{equation}

 We note that L\"uhrmann-Mendelson \cite{LM} also 
considered a similar randomization of the form \eqref{R1} in the study of NLW on $\R^3$. See Remark \ref{REM:LM} below. The randomization in \cite{LM} stems from yet another one  used by Zhang and Fang \cite{ZF} in their study of the Navier-Stokes equations.
We point out, however,  that the main purpose of our paper is to explain
how the randomization of the form \eqref{R1} is naturally associated to the Wiener decomposition and hence the modulation spaces. See also our previous paper \cite{BP} in the periodic setting. Thus, from the perspective of time-frequency analysis, the Wiener randomization seems to be the ``right" one.

In the sequel, we make the following assumption; there exists $c>0$ such that
\begin{equation}
\bigg| \int_{\R} e^{\g x } d \mu_n^{(j)}(x) \bigg| \leq e^{c\g^2}
\label{R2}
\end{equation}
	
\noi
for all $\g \in \R$, $n \in \Z^d$, $j = 1, 2$.
Note that \eqref{R2} is satisfied by
standard complex-valued Gaussian random variables,
standard Bernoulli random variables,
and any random variables with compactly supported distributions.

It is easy to see that, if $\phi \in H^s(\R^d)$,
then  the randomized function $\phi^\o$ is
almost surely  in $H^s(\R^d)$; see Lemma \ref{LEM:Hs} below.
One can also show that there is no smoothing upon randomization
in terms of differentiability;
see, for example, Lemma B.1 in \cite{BT2}.
Instead, the main point of this randomization is
its improved integrability;
 if $\phi \in L^2(\R^d)$,
then  the randomized function $\phi^\o$ is
almost surely  in $L^p(\R^d)$ for any finite $p \geq 2$;
see Lemma \ref{LEM:Lp} below. Such results for random Fourier series
are known as  Paley-Zygmund's theorem \cite{PZ};
see also Kahane's book \cite{Kahane} and Ayache-Tzvetkov \cite{AT}.

\begin{remark}\rm
One may fancy a randomization associated to Besov spaces,
of the form:
\begin{equation*}
\phi^\omega : = \sum_{j = 0}^\infty g_n (\omega) \varphi(D) \phi.
\end{equation*}

\noi
In view of the Littlewood-Paley theory,
such a randomization does not yield any improvement
on differentiability or integrability
and thus it is of no interest.
\end{remark}

\subsection{Main results.} The Wiener randomization of an initial condition allows us to establish some improvements
of the Strichartz estimates. In turn, these probabilistic Strichartz estimates yield an almost sure well-posedness result for NLS.
First, we recall the usual Strichartz estimates on $\R^d$ for the reader's convenience.
We say that
a pair  $(q, r)$ is \emph{Schr\"odinger admissible}
if it satisfies
\begin{equation}
\frac{2}{q} + \frac{d}{r} = \frac{d}{2}
\label{Str2}
\end{equation}

\noi
 with $2\leq q, r \leq \infty$
and $(q, r, d) \ne (2, \infty, 2)$.
Let $S(t)= e^{it\Dl}$.
Then, the following Strichartz estimates
are known to hold.

\begin{lemma}[\cite{Strichartz, Yajima, GV, KeelTao}]\label{LEM:Str}
Let $(q, r)$ be  Schr\"odinger admissible.
Then, we have 
\begin{equation}
\| S(t) \phi\|_{L^q_t L^r_x (\R\times \R^d)} \lesssim \|\phi\|_{L^2_x(\R^d)}.
\label{Str1}
\end{equation}

\end{lemma}

Next, we present improvements
of the Strichartz estimates under the Wiener randomization.
Proposition \ref{PROP:R2}
will be then used for a local-in-time theory,
while Proposition \ref{PROP:R5}
is useful for small data global theory.
The proofs of Propositions \ref{PROP:R2} and \ref{PROP:R5} are presented
in Section 2.

\begin{proposition}[Improved local-in-time  Strichartz estimate]\label{PROP:R2}
Given $\phi \in L^2(\R^d)$,
let $\phi^\o$ be its randomization defined in \eqref{R1}, satisfying \eqref{R2}.
Then,
given $2\leq q, r<\infty$,
there exist $C, c>0$ such that
\begin{align}
P\Big( \|S(t) \phi^\omega\|_{L^q_t L^r_x([0, T]\times \R^d)}> \ld\Big)
\leq C\exp\bigg(-c\frac{\ld^2}{T^{\frac{2}{q}}\|\phi\|_{L^2}^{2}}\bigg)
\label{R2aa}
\end{align}
	
\noi
for all $T > 0$ and $\ld > 0$.
\end{proposition}
\noi

In particular,
by setting $\ld = T^\theta \|\phi\|_{L^2}$, we have
\begin{equation*}
\|S(t) \phi^\o\|_{L^q_tL^r_x([0, T]\times \R^d)}
\les T^\theta \|\phi\|_{L^2(\R^d)}
\end{equation*}

\noi
outside a set of probability at most
$ C \exp \big(-c T^{2\theta - \frac{2}{q}}\big).$
Note that, as long as $\theta < \frac{1}{q}$, this probability can be made arbitrarily small by letting $T\to 0$.
Moreover, for fixed $T>0$, we have the following;
given any small $\eps>0$, we have
\[\|S(t) \phi^\o\|_{L^q_tL^r_x([0, T]\times \R^d)}
\leq C_T \Big( \log \frac{1}{\eps}\Big)^{\frac{1}{2}} \|\phi\|_{L^2}\]

\noi
outside a set of probability $< \eps$.

The next proposition states an improvement of the Strichartz estimates in the global-in-time setting.

\begin{proposition}[Improved global-in-time Strichartz estimate]\label{PROP:R5}
Given $\phi \in L^2(\R^d)$,
let $\phi^\o$ be its randomization defined in \eqref{R1}, satisfying \eqref{R2}.
Given a Schr\"odinger admissible pair $(q, r)$ with $q, r < \infty$,
let $\wt {r} \geq r$.
Then, there exist $C, c>0$ such that
\begin{align}
P\Big( \|S(t) \phi^\omega\|_{L^q_t L^{\wt{r}}_x ( \R \times \R^d)} > \ld\Big)
\leq Ce^{-c \ld^2 \|\phi\|_{L^2}^{-2}}
\label{R5a}
\end{align}

\noi
for all $\ld > 0$.
In particular, given any small $\eps > 0$, we have
\[ \|S(t) \phi^\omega\|_{L^q_t L^{\wt{r}}_x ( \R \times \R^d)}
\les \Big( \log \frac{1}{\eps}\Big)^\frac{1}{2} \|\phi\|_{L^2} \]

\noi
outside a set of probability at most $\eps$.

\end{proposition}

We conclude this introduction by
discussing an example of
almost sure local well-posedness of NLS
with randomized initial data below a scaling critical  regularity.
In the following, we consider
the energy-critical cubic NLS on $\R^4$:
\begin{equation}
\label{NLS4d}i \partial_t u + \Delta u = \pm |u|^{2}u, \quad
\qquad (t, x) \in \R \times \R^4.
\end{equation}

\noi
Cazenave-Weissler \cite{CW}
proved local well-posedness of \eqref{NLS4d}
with initial data in the critical space $\dot H^1(\R)$.
See
Ryckman-Vi\c{s}an \cite{RV}, Vi\c{s}an \cite{Visan},  and  Kenig-Merle \cite{KM} for global-in-time results.
In the following, we state a local well-posedness result
of \eqref{NLS4d} with random initial data below the critical space.
More precisely,
given  $\phi \in H^s(\R^4) \setminus H^1(\R^4)$, $s \in( \frac {3}{5}, 1)$,
and  its randomization $\phi^\o$ defined in \eqref{R1},
we prove that \eqref{NLS4d} is almost surely locally well-posed
with random initial data $\phi^\o$.
Although $\phi$ and its randomization $\phi^\o$ lie
in a supercritical regularity regime, the Wiener randomization essentially makes the  problem {\it subcritical}. This is a common feature for many of the probabilistic well-posedness results.

\begin{theorem}\label{THM:1}
Let $s \in (\frac{3}{5}, 1)$.
Given $\phi \in H^s(\R^4)$, let $\phi^\o$ be its randomization defined in \eqref{R1}, satisfying \eqref{R2}.
Then, the cubic NLS \eqref{NLS4d} on $\R^4$
is almost surely locally well-posed
with respect to the randomization $\phi^\omega$ as initial data.
More precisely,
there exist $ C, c, \g>0$ and $\s = 1+$ such that
for each $T\ll 1$,
there exists a set $\O_T \subset \O$ with the following properties:

\smallskip
\begin{itemize}
\item[\textup{(i)}]
$P(\O\setminus\O_T) \leq C \exp\Big(-\frac{c}{T^{\g} \|\phi\|_{H^s}^{2}}\Big)$,

\item[\textup{(ii)}]
For each $\o \in \O_T$, there exists a (unique) solution $u$
to \eqref{NLS4d}
with $u|_{t = 0} = \phi^\o$
in the class
\[ S(t) \phi^\o + C([-T, T]: H^{\s} (\R^4)) \subset C([-T, T]:H^s(\R^4)).\]

\end{itemize}
\end{theorem}

\noi
The details of the proof of Theorem~\ref{THM:1} are presented in Section 3. We discuss here a very brief outline of the argument. Denoting the linear and nonlinear parts of $u$ by
$z (t) = z^\o(t) : = S(t) \phi^\o$
and $v(t) := u(t) - S(t) \phi^\o$ respectively,
we can reduce \eqref{NLS4d} to
\begin{equation}
\begin{cases}
	 i \dt v + \Dl v = \pm |v + z|^2(v+z)\\
v|_{t = 0} = 0.
 \end{cases}
\label{NLS4d2}
\end{equation}

\noi
We then prove that
the Cauchy problem \eqref{NLS4d2}	
is almost surely locally well-posed for $v$,
viewing $z$ as a random forcing term.
This is done by using the standard subcritical $X^{s, b}$-spaces with $b > \frac 12$
defined by
\[ \|u\|_{X^{s, b}(\R\times \R^4)} = \| \jb{\xi}^s \jb{\tau + |\xi|^2}^b \ft{u}(\tau, \xi)\|_{L^2_{\tau, \xi}(\R\times\R^4)}.\]

\noi
We point out that  the uniqueness in Theorem \ref{THM:1}
refers to  uniqueness of the nonlinear part $v(t) = u(t) - S(t) \phi^\o$ of a solution $u$.

We conclude this introduction with several remarks.

\begin{remark}\rm

Theorem \ref{THM:1} holds for both defocusing and focusing cases
(corresponding to the $+$ sign and the $-$ sign in \eqref{NLS1}, respectively)
due to the local-in-time nature of the problem.
\end{remark}

\begin{remark}\rm
Theorem \ref{THM:1} can also be proven
with the variants of the $X^{s, b}$-spaces
adapted to the $U^p$- and $V^p$-spaces
introduced by Koch, Tataru, and their collaborators \cite{KochT, HHK, HTT11}.
These spaces are designed to handle  problems in critical regularities.
We decided to present the proof with the
classical subcritical $X^{s, b}$-spaces, $b > \frac 12$,
to emphasize that the problem has become subcritical
upon randomization.
We should, however, point out that,
with the spaces introduced by Koch and Tataru,
we can also prove probabilistic  small data global
well-posedness and scattering
as a consequence of the probabilistic  global-in-time Strichartz estimates (Proposition \ref{PROP:R5}).
See our paper \cite{BOP2}
for an example of such results for the cubic NLS on $\R^d$, $d \geq 3$.

It is of interest to consider almost sure global existence for \eqref{NLS4d}.
While the mass of $v$ in \eqref{NLS4d2} has a global-in-time control,
there is no
 energy conservation for $v$
 and thus  we do not know how to proceed at this point.
In  \cite{BOP2}, we establish
almost sure global existence for \eqref{NLS4d},
assuming an a priori control on the $H^1$-norm of the nonlinear part $v$ of a solution.
We
also prove there, without any assumption,
global existence with a large probability
by considering a randomization, not on unit cubes
but on dilated cubes this time.

In the context of the defocusing cubic NLW on $\R^4$,
one can obtain an a priori bound on the energy of the nonlinear part of a solution, see \cite{BT3}. As a consequence, the third author \cite{POC} proved almost sure global well-posedness
of the energy-critical defocusing cubic NLW on $\R^4$
below the scaling critical regularity.

\end{remark}

\begin{remark}\rm
In Theorem \ref{THM:1},
we simply used $\s = 1+$ as the regularity of
the nonlinear part $v$.
It is possible to characterize the possible values of $\s$ in terms of
the regularity $s < 1$ of $\phi$.
However, for simplicity of  presentation, we omitted such a  discussion.
\end{remark}

\begin{remark}\rm
In probabilistic well-posedness results \cite{Bo2, Bo97, CO, NS} for NLS on $\T^d$,
random initial data are assumed to be of the following specific form:
\begin{equation}
\label{I3}
 u_0^\omega(x) = \sum_{n \in \Z^d} \frac{g_n(\o)}{(1+|n|^2)^\frac{\al}{2}}e^{i n \cdot x},
 \end{equation}

\noi
where $\{g_n\}_{n \in \Z^d}$ is a sequence of independent
complex-valued standard Gaussian random variables.
The expression \eqref{I3}
has a close connection to the study of invariant (Gibbs) measures
and, hence, it is of importance.
At the same time, due to the lack of a full range of Strichartz estimates on $\T^d$, one could not handle a general randomization of a given function as in \eqref{I2}. In Theorem \ref{THM:1}, we consider NLS on $\R^4$ and thus we do not encounter this issue thanks to a full range of the Strichartz estimates.
For NLW, finite speed of propagation allows us to use a full range of Strichartz estimates even on compact domains, at least locally in time; thus, in that context, one does not encounter such an issue.

\end{remark}

\begin{remark}\label{REM:LM}\rm
In a recent preprint,
L\"uhrmann-Mendelson \cite{LM}
considered the defocusing NLW on $\R^3$
with randomized initial data defined in \eqref{R1}
in a supercritical regularity
and proved almost sure global well-posedness in the energy-subcritical case,
following the method developed in \cite{CO}.
For the energy-critical quintic NLW, they obtained
almost sure local well-posedness
along with small data global existence and scattering.

\end{remark}

\section{Probabilistic Strichartz estimates}

In this section, we state and prove some basic properties
of the randomized function $\phi^\o$ defined in \eqref{R1},
including the improved Strichartz estimates (Propositions \ref{PROP:R2}
and \ref{PROP:R5}).
First, recall the following probabilistic estimate. See \cite{BT2} for the proof.

\begin{lemma} \label{LEM:R1}
Assume \eqref{R2}. Then, there exists $C>0$ such that
\[ \bigg\| \sum_{n \in \Z^d} g_n(\omega) c_n\bigg\|_{L^p(\Omega)}
\leq C \sqrt{p} \| c_n\|_{\l^2_n(\Z^d)}\]

\noi
for all $p \geq 2$ and $\{c_n\} \in \l^2(\Z^d)$.
\end{lemma}

Given $\phi \in H^s$, it is easy to see that its randomization $\phi^\o \in H^s$
almost surely, for example,
if $\{g_n\}$ has a uniform finite variance.
Under the assumption \eqref{R2}, we have
a more precise description on the size of $\phi^\o$.

\begin{lemma} \label{LEM:Hs}
Given  $\phi \in H^s(\R^d)$, let $\phi^\o$ be its randomization defined in \eqref{R1},  satisfying \eqref{R2}. Then, we have
\begin{align}
P\Big( \| \phi^\omega \|_{ H^s(  \R^d)} > \ld\Big)
\leq C e^{-c \ld^2  \|\phi\|_{ H^s}^{-2}}
\label{Hs1}
\end{align}

\noi
for all $\ld > 0$.
\end{lemma}

\begin{proof}
By Minkowski's integral inequality and Lemma \ref{LEM:R1}, we have
\begin{align*}
\Big(\mathbb{E} \|  \phi^\o   \|_{ H^s(\R^d)}^p\Big)^\frac{1}{p}
&  \leq \big\| \|\jb{\nb}^s \phi^\o \|_{L^p(\Omega)} \big\|_{L^2_{ x}( \R^d)}
\les \sqrt p \big\|
\|\psi(D-n) \jb{\nb}^s \phi  \|_{\l^2_n}\big\|_{L^2_{x}} \notag \\
&
\sim  \sqrt p \| \phi \|_{ H^s}
\end{align*}

\noi
for any $p \geq 2$. Thus, we have obtained
\[ \mathbb{E}[ \|\phi^\o\|_{H^s}^p] \leq C_0^p p^\frac{p}{2} \|\phi\|_{H^s}^p.
\]

\noi
By Chebyshev's inequality, we have
\begin{align}
P\Big( \|\phi^\o\|_{H^s} > \ld \Big) 
<  \bigg(\frac{C_0 p^\frac{1}{2} \|\phi\|_{H^s}}{\ld}\bigg)^p
\label{Cheby}
\end{align}

\noi
for $p \geq 2$.

Let $p =\Big( \frac{\ld}{C_0e\|\phi\|_{H^s}}\Big)^2$.
If $p \geq 2$, then by \eqref{Cheby}, we have
\begin{align*}
P\Big( \|\phi^\o\|_{H^s} > \ld \Big) < \bigg(\frac{C_0 p^\frac{1}{2} \|\phi\|_{H^s}}{\ld}\bigg)^p
= e^{-p} = e^{-c \ld^2\|\phi\|_{H^s}^{-2}}.
\end{align*}

\noi
Otherwise, i.e.~if  $p =\Big( \frac{\ld}{C_0e\|\phi\|_{H^s}}\Big)^2 \leq 2$,
we can choose $C$ such that $C e^{-2} \geq 1$.
Then, we have
\begin{align*}
P\Big( \|\phi^\o\|_{H^s} > \ld \Big) \leq 1
\leq C e^{-2} \leq C e^{-c \ld^2\|\phi\|_{H^s}^{-2}},
\end{align*}
thus giving the desired result.
\end{proof}

The next lemma shows that, if $\phi \in L^2(\R^d)$,
then its randomization $\phi^\o$ is almost surely in $L^p(\R^d)$ for any $p\in [2, \infty)$.

\begin{lemma} \label{LEM:Lp}
Given  $\phi \in L^2(\R^d)$, let $\phi^\o$ be its randomization defined in \eqref{R1},
 satisfying \eqref{R2}.
Then, given finite $p \geq 2$, there exist $C, c >0$ such that
\begin{align}
P\Big( \| \phi^\omega \|_{ L^p(  \R^d)} > \ld\Big)
\leq C e^{-c \ld^2  \|\phi\|_{ L^2}^{-2}}
\label{Lp}
\end{align}

\noi
for all $\ld > 0$.
In particular, $\phi^\o$ is in $L^p(\R^d)$ almost surely.
\end{lemma}

\begin{proof}
By Lemma \ref{LEM:R1}, we have
\begin{align*}
\Big(\mathbb{E} \|  \phi^\o   \|_{ L^p_x (\R^d)}^r\Big)^\frac{1}{r}
&  \leq \big\| \| \phi^\o \|_{L^r(\Omega)} \big\|_{L^p_{ x}( \R^d)}
\les \sqrt r \big\|
\|\psi(D-n)  \phi  \|_{\l^2_n}\big\|_{L^p_{x}} \notag \\
&
\leq \sqrt r \big\|
\|\psi(D-n)  \phi  \|_{L^p_x}\big\|_{\l^2_n}
\leq \sqrt r \big\|
\|\psi(D-n)  \phi  \|_{L^2_x}\big\|_{\l^2_n} \notag \\
& \sim  \sqrt r \| \phi \|_{L^2_x}
\end{align*}

\noi
for any $r \geq p$.
Then, \eqref{Lp} follows as in the proof of Lemma~\ref{LEM:Hs}.
\end{proof}

We conclude this section by presenting  the proofs of the improved Strichartz estimates under randomization.
Before continuing further, we briefly recall the definitions of the smooth projections from Littlewood-Paley theory.
Let $\varphi$ be a smooth real-valued bump function supported on $\{\xi\in \R^d: |\xi|\leq 2\}$ and $\varphi\equiv 1$ on $\{\xi: |\xi|\leq 1\}$. If $N>1$ is a dyadic number, we define the smooth projection $\P_{\leq N}$ onto frequencies
$\{|\xi| \leq N\}$ by
\[\widehat{\P_{\leq N}f}(\xi):=\varphi\big(\tfrac{\xi}N\big)\widehat f(\xi).\]

%
\noi
Similarly, we can define the smooth projection $\P_{N}$ onto frequencies $\{|\xi|\sim  N\}$ by
\[\widehat{\P_{N}f}(\xi):=\Big(\varphi\big(\tfrac{\xi}N\big)-\varphi\big(\tfrac{2\xi}N\big)\Big)\widehat f(\xi).\]
We make the convention that $\P_{\leq 1}=\P_1$. Bernstein's inequality states that
\begin{equation}
\|\P_{\leq N} f\|_{L^q(\R^d)} \les N^{\frac dp-\frac dq}\|\P_{\leq N} f\|_{L^p(\R^d)}, \qquad
1\leq p \leq q \leq \infty.
\label{R3}
\end{equation}

\noi The same inequality holds if we replace $\P_{\leq N}$ by $\P_N$. As an immediate corollary, we have
\begin{equation}
\|\psi(D -n) \phi\|_{L^q(\R^d)} \les \|\psi(D-n)  \phi \|_{L^p(\R^d)}, \qquad
1\leq p \leq q \leq \infty,
\label{R4}
\end{equation}

\noi
for all $n \in \Z^d$.
This follows from applying \eqref{R3}
to $\phi_n(x) := e^{2\pi i n\cdot x} \psi(D-n) \phi(x)$
and noting that $\supp \ft \phi_n \subset [-1, 1]^d$ .
The point of \eqref{R4} is that once a function is (roughly) restricted to a cube, we do not need to lose any regularity to go from the $L^q$-norm  to the $L^p$-norm, $q \geq p$.

\begin{proof}[Proof of Proposition \ref{PROP:R2}]

Let $q, r \geq 2$. We write $L_T^q$ to denote $L^q_t([0, T])$.
By Lemma \ref{LEM:R1} and \eqref{R4}, we have
\begin{align*}
\Big(\mathbb{E} \|S(t)  \phi^\omega & \|_{L^q_t L^r_x( [0, T] \times \R^d)}^p\Big)^\frac{1}{p}
  \leq \Big\| \| S(t) \phi^\omega\|_{L^p(\Omega)} \Big\|_{L^q_T L^r_x}
\leq \sqrt{p}\Big\| \| \psi(D-n) S(t) \phi\|_{\l^2_n } \Big\|_{L^q_T L^r_x}\\
& \leq \sqrt{p}\Big\| \| \psi(D-n) S(t) \phi\|_{L^r_x } \Big\|_{ L^q_T \l^2_n }
 \les \sqrt{p}\Big\| \| \psi(D-n) S(t) \phi\|_{L^2_x } \Big\|_{ L^q_T \l^2_n}\\
&  \les T^\frac{1}{q}
 \sqrt{p}\|\phi\|_{L^2_x}
\end{align*}

\noi
for $p \geq \max (q, r)$.
Then, \eqref{R2aa} follows as in the proof of Lemma~\ref{LEM:Hs}.
\end{proof}

\begin{proof}[Proof of Proposition \ref{PROP:R5}]
Let $(q, r)$ be Schr\"odinger admissible and $\wt{r} \geq r$.
Then, proceeding as before, 
 we have
\begin{align*}
\Big(\mathbb{E} \|S(t) \phi^\omega \|_{L^q_t L^{\wt r}_x( \R \times \R^d)}^p\Big)^\frac{1}{p}
& \les \sqrt p \Big\|
\|\psi(D-n)S(t) \phi\|_{L^{\wt r}_x}\Big\|_{\l^2_n L^q_{t}}
 \les \sqrt p \Big\|
\|\psi(D-n)S(t) \phi\|_{ L^q_{t} L^{r}_x}\Big\|_{\l^2_n} \notag \\
\intertext{By Lemma \ref{LEM:Str}, }
&
\les
 \sqrt p \big\|
\|\psi(D-n)  \phi  \|_{L^2_x}\big\|_{\l^2_n}
\sim
 \sqrt p
\| \phi  \|_{L^2_x}
\end{align*}
	
\noi
for $p \geq \max(q, \wt{r})$.
Finally, \eqref{R5a} follows as above.
\end{proof}

\begin{remark}\rm
The Cauchy problem  \eqref{NLS1}
has also been studied
for initial data in the modulation spaces $M_s^{p, 1}$ for $ 1 \leq p \leq \infty$ and $s\geq 0$; see, for example, \cite{BO} and \cite{wahahu}.
Thus, it is tempting to consider what happens if we randomize an initial condition  in a modulation space $M^{p, q}_s$. In this case, however, there is no improvement in the Strichartz estimates in terms of integrability, i.e.~$p$, hence, no improvement of well-posedness with respect to $M^{p, q}_s$ in terms of differentiability, i.e.~in $s$.
Indeed, 
by computing the moments of the modulation norm of the randomized function \eqref{R1}, one immediately sees that the modulation norm  remains essentially unchanged
due to the outside summation over $n$. In the proof of Propositions \ref{PROP:R2} and \ref{PROP:R5}, the averaging effect of a linear combination of the random variables $g_n$ played a crucial role.
For the modulation spaces, we do not have such an averaging effect
since the outside summation over $n$ forces us to work on a piece restricted to each cube, i.e. each random variable at a time.
\end{remark}

\section{Almost sure local well-posedness}

Given $\phi \in H^s(\R^d)$,
let $\phi^\o$ be its randomization defined in \eqref{R1}.
In the following, we consider
the Cauchy problem \eqref{NLS4d}
with random initial data $u|_{t = 0} = \phi^\o$.
By letting
$z (t) = z^\o(t) : = S(t) \phi^\o$
and $v(t) := u(t) - S(t) \phi^\o$,
we can reduce \eqref{NLS1} to
\begin{equation}
\begin{cases}
	 i \dt v + \Dl v = \pm |v + z|^2(v+z)\\
v|_{t = 0} = 0.
 \end{cases}
\label{NLS2}
\end{equation}

\noi
By expressing \eqref{NLS2} in the Duhamel formulation,
we have
\begin{equation}
v(t) =  \mp i\int_0^t S(t-t') \N (v+ z)(t') dt',
\label{NLS3}
\end{equation}

\noi
where $\N(u) = |u|^2 u=u\bar u u$.
Let $\eta$ be a smooth cutoff function supported on $[-2, 2]$,
$\eta \equiv 1$ on $[-1, 1]$, and let $\eta_{_T}(t) = \eta\big(\frac{t}{T}\big)$.
Note that,
if $v$ satisfies
\begin{equation}
v(t) =  \mp i\eta(t)\int_0^t S(t-t') \eta_{_T}(t')\N (\eta v+ \eta_{_T}z)(t') dt'
\label{NLS4}
\end{equation}

\noi
\noi
for some $T \ll 1$,
then
it also satisfies \eqref{NLS3} on $[-T, T]$.
Hence, 
we consider \eqref{NLS4} in the following.

Given $z(t) = S(t) \phi^\o$, define $\G$ by
\begin{equation}
\G v(t) =  \mp i \eta(t)\int_0^t S(t-t')\eta_{_T}(t') \N (\eta v+ \eta_{_{T}}z)(t') dt'.
\label{NLS5}
\end{equation}

\noi
Then, the following nonlinear estimates yield Theorem \ref{THM:1}.

\begin{proposition}\label{PROP:LWP}
Let $s\in\big(\frac 35, 1\big)$. Given $\phi \in H^s(\R^4)$,
let $\phi^\o$ be its randomization defined in \eqref{R1}, satisfying \eqref{R2}.
Then,
there exists $\s = 1+$, $ b = \frac{1}{2}+$ and $\theta = 0+$ such that
for each small $T\ll 1$ and $R>0$, we have
\begin{align}
\|\G v\|_{X^{\s, b}}
& \leq C_1T^\theta
(\|v\|_{X^{\s, b}}^3 + R^3),  \label{nl0a}\\
\|\G v_1 - \G v_2  \|_{X^{\s, b}}
& \leq C_2 T^\theta
(\|v_1\|_{X^{\s, b}}^2+\|v_2\|_{X^{\s, b}}^2 + R^2)
\|v_1 -v_2 \|_{X^{\s, b}},
\label{nl0b}
\end{align}

\noi
outside a set of probability at most $ C \exp\Big(-c \frac{R^2}{\|\phi\|_{H^s}^{2}}\Big)$.

\end{proposition}

We first present the proof of Theorem \ref{THM:1}, assuming Proposition \ref{PROP:LWP}.
Then, we prove Proposition \ref{PROP:LWP}
at the end of this section.

\begin{proof}[Proof of Theorem \ref{THM:1}]
Let $B_1$ denote the ball of radius 1
centered at the origin in $X^{\s, b}$.
Then,
given $T \ll 1$,  we show that the map $\G$ is a contraction on $B_1$.
Given $T \ll 1$,
we choose $R = R(T) \sim T^{-\frac{\g}{2}}$
for some $\g \in (0, \frac{2\theta}{3})$
such that
\[ C_1 T^\theta(1+ R^3) \leq 1\quad \text{and} \quad C_2 T^\theta(2+R^2) \leq \frac 12. \]

\noi
Then, for $v, v_1, v_2 \in B_1$, Proposition \ref{PROP:LWP} yields
\begin{align*}
\|\G v\|_{X^{\s, b}}
& \leq 1, \\
\|\G v_1 - \G v_2  \|_{X^{\s, b}}
& \leq  \frac {1}{2}
\|v_1 -v_2 \|_{X^{\s, b}}
\end{align*}

\noi
outside an exceptional set  of probability at most
\[ C \exp\bigg(-c \frac{R^2}{\|\phi\|_{H^s}^{2}}\bigg)
= C \exp \bigg(-\frac{c}{T^{\g}\|\phi\|_{H^s}^{2}}\bigg).\]

\noi
Therefore, by defining $\O_T$ to be the complement of this exceptional set,
it follows that, for $\o\in \O_T$, there exists a unique $v^\o \in B_1$
such that $\G v^\o = v^\o$.
This completes the proof of Theorem \ref{THM:1}.
\end{proof}

Hence, it remains to prove Proposition \ref{PROP:LWP}.
Before proceeding further, we first present some lemmata
on the basic $X^{s, b}$-estimates.
See  \cite{Bo2, KPV93, TAO}
for the basic properties of the $X^{s, b}$-spaces.

\begin{lemma}\label{LEM:Xsb1}
\textup{(i)} Linear estimates:
Let $T \in( 0, 1)$ and $b \in \big(\frac{1}{2}, \frac{3}{2}\big]$.
Then, for $s \in \R$ and $\theta \in \big[0, \frac{3}{2} - b\big)$, we have
\begin{align}
\|  \eta_{_T} (t) S(t) \phi \|_{X^{s, b}(\R\times \R^4)}
& \les T^{\frac{1}{2}-b}\|\phi\|_{H^s(\R^4) }, \label{Xsb1}\\
\bigg\| \eta(t)
\int_0^t S(t - t') \eta_{_T}(t') F(t') dt' \bigg\|_{X^{s, b}(\R\times \R^4)}
& \les T^{\theta} \|F\|_{X^{s, b -1+\theta}(\R\times \R^4)}.\notag
\end{align}

\noi
\textup{(ii)} Strichartz estimates:
Let $(q, r)$ be Schr\"odinger admissible
and $p \geq 3$. Then, for $ b > \frac{1}{2}$
and $N_1 \leq N_2$, we have
\begin{align}
\|  u \|_{L^q_t L^r_x(\R\times \R^4)} & \les \|u\|_{X^{0, b}(\R\times \R^4)}, \label{Xsb2}\\
\|  u \|_{L^p_{t, x}(\R\times \R^4)}
& \les  \big\||\nb|^{2 - \frac 6p} u\big\|_{X^{0, b}(\R\times \R^4)},
\label{Xsb3}\\
\| \P_{N_1} u_1 \P_{N_2}u_2\|_{L^2_{t, x}(\R\times \R^4)}
& \les N_1 \bigg(\frac{N_1}{N_2}\bigg)^{\frac 12}
\|\P_{N_1} u_1\|_{X^{0, b}(\R\times\R^4)}\|\P_{N_2} u_2\|_{X^{0, b}(\R\times\R^4)}.
\label{Xsb4}
\end{align}

\end{lemma}
	
\noi
Recall that
\eqref{Xsb2} follows from the Strichartz estimate \eqref{Str1}
and
\eqref{Xsb3} follows from
Sobolev inequality and  \eqref{Str1},
while
\eqref{Xsb4} follows from a refinement of the Strichartz estimate
by Bourgain \cite{Bo98} and Ozawa-Tsutsumi \cite{OT}.

As a corollary to Lemma \ref{LEM:Xsb1}, we have the following estimates.
\begin{lemma}\label{LEM:Xsb2}
Given small $\eps > 0$, let $\eps_1 = 2\eps+$.
Then, for $N_1 \leq N_2$, we have
\begin{align}
\|  u \|_{L^\frac{3}{1+\eps_1}_{t, x}(\R\times \R^4)} & \les \|u\|_{X^{0, \frac12 - 2\eps }(\R\times \R^4)}, \label{Xsb5}\\
\| \P_{N_1}  u_1 \P_{N_2}u_2\|_{L^2_{t, x}(\R\times \R^4)} 
& \les
C(N_1, N_2)
 \|\P_{N_1}u_1\|_{X^{0, \frac{1}{2}+}(\R\times\R^4)}\|\P_{N_2} u_2\|_{X^{0, \frac{1}{2}-2\eps}(\R\times\R^4)},
\label{Xsb6}
\end{align}

\noi
where $C(N_1, N_2)$ is given by
\[
C(N_1, N_2) =
\begin{cases}
N_1^{\frac{3}{2} + \eps_1 +}
N_2^{-\frac 12 + \eps_1} & \text{if } N_1 \leq N_2, \\
N_1^{- \frac{1}{2} + 5\eps_1 +}
N_2^{\frac 32 -  3 \eps_1} & \text{if } N_1 \geq N_2.
\end{cases}
\]

\end{lemma}

\noi
\begin{proof}
The first estimate \eqref{Xsb5} follows from interpolating \eqref{Xsb2} with $q = r =3$
and $\| u\|_{L^2_{t, x}} = \|u\|_{X^{0, 0}}.$
The second estimate \eqref{Xsb6} follows from interpolating \eqref{Xsb4}
and
\[
\| \P_{N_1} u_1 \P_{N_2}u_2\|_{L^2_{t, x}}
 \leq\| \P_{N_1} u_1\|_{L^\infty_{t, x}}\| \P_{N_2}u_2\|_{L^2_{t, x}}
 \les \| \P_{N_1} u_1\|_{X^{2 +, \frac 12+}}\| \P_{N_2}u_2\|_{X^{0, 0}}.
\qedhere \]
%
\end{proof}

We are now ready to prove Proposition \ref{PROP:LWP}.

\begin{proof}[Proof of Proposition \ref{PROP:LWP}]
We only prove \eqref{nl0a} since  \eqref{nl0b} follows in a  similar manner.
By Lemma \ref{LEM:Xsb1} (i) and duality, we have
\begin{align}
\|\G v(t)\|_{X^{\s, b}}
& \les T^\theta \| \N
(\eta v+ \eta_{_{T}}z) \|_{X^{\s, b-1+\theta}} \notag \\
& = T^\theta \sup_{\|v_4\|_{X^{0, 1-b-\theta}}\leq1} \bigg|\iint_{\R\times \R^4}
\jb{\nb}^\s \big[ \N (\eta v+ \eta_{_{T}}z) \big] v_4 dx dt\bigg|.
\label{nl1}
\end{align}

\noi
We estimate the right-hand side of \eqref{nl1}
by performing  case-by-case analysis of expressions of the form:
\begin{align}
\bigg| \iint_{\R\times \R^4}
\jb{\nb}^\s ( w_1 w_2 w_3 )v_4 dx dt\bigg|,
\label{nl2}
\end{align}

\noi
where $\|v_4\|_{X^{0, 1-b-\theta}} \leq 1$
and $w_j= \eta v$ or $\eta_{_{T}} z$,  $j = 1, 2, 3$.
Before proceeding further, let us simplify some of the notations.
In the following, we drop the complex conjugate sign
since it plays no role.
Also, we often suppress
the smooth cutoff function $\eta$ (and $\eta_{_{T}}$) from
$w_j = \eta v$ (and $w_j = \eta_{_{T}} z$)
and simply denote them by $v_j$ (and $z_j$,  respectively).
Lastly, in most of the cases,
we dyadically decompose
 $w_j$, $j = 1, 2, 3$,
and $v_4$ such that their spatial frequency supports are $\{ |\xi_j|\sim N_j\}$
for some dyadic $N_j \geq 1$
but still denote them by $w_j$, $j = 1, 2, 3$, and $v_4$.

Let $b = \frac{1}{2} + \eps$
and $\theta = \eps $
for some small $\eps > 0$ (to be chosen later)
so that
$ 1- b - \theta = \frac 12 - 2 \eps$.
In the following, we set $\eps_1 = 2\eps+$.

\medskip
\noi
{\bf  Case (1):} $v v v$ case.

In this case, we do not need to  perform dyadic decompositions
and we divide the frequency spaces into
$\{|\xi_1| \geq |\xi_2|, |\xi_3|\}$,
$\{|\xi_2| \geq |\xi_1|, |\xi_3|\}$,
and
$\{|\xi_3| \geq |\xi_1|, |\xi_2|\}$.
Without loss of generality, assume that $|\xi_1| \ges |\xi_2|, |\xi_3|$.
By $L^3L^{\frac{6}{1-\eps_1}}L^{\frac{6}{1-\eps_1}}L^{\frac{3}{1+\eps_1}}$-H\"older's inequality
and Lemmata \ref{LEM:Xsb1} and  \ref{LEM:Xsb2},
we have
\begin{align*}
\bigg|\int_{\R\times \R^4} \jb{\nb}^\s v_1 v_2 v_3 v_4 dx dt \bigg|
& \leq \| \jb{\nb}^\s v_1\|_{L^3_{t, x}} \|v_2\|_{L^{\frac{6}{1-\eps_1}}_{t, x}}\|v_3\|_{L^{\frac{6}{1-\eps_1}}_{t, x}}
\|v_4\|_{L^{\frac{3}{1+\eps_1}}_{t, x}}\\
& \les
\prod_{j = 1}^3 \|v_j\|_{X^{\s, \frac{1}{2}+}} \|v_4\|_{X^{0, \frac{1}{2}-2\eps}}
\les \prod_{j = 1}^3 \|v_j\|_{X^{\s, b}}
\end{align*}

\noi
for $\s \geq 1 + \eps_1 = 1+ 2\eps+$.

\medskip

\noi
{\bf Case (2):} $zz z$ case. \quad

Without loss of generality, assume $N_3 \geq N_2 \geq N_1$.

\smallskip

\noi
{\bf $\bullet$  Subcase (2.a):} $N_2 \sim N_3$.

By $L^{\frac{6}{1-2\eps_1}}L^4L^4L^{\frac{3}{1+\eps_1}}$-H\"older's inequality
and
Lemmata \ref{LEM:Xsb1} and  \ref{LEM:Xsb2},
we have
\begin{align*}
\bigg|\int_{\R\times \R^4}  z_1 z_2 \jb{\nb}^\s z_3 v_4 dx dt \bigg|
& \les \|z_1\|_{L^\frac{6}{1-2\eps_1}_{t, x}} \|\jb{\nb}^\frac{\s}{2}z_2\big\|_{L^4_{t, x}}
\|\jb{\nb}^\frac{\s}{2} z_3 \|_{L^4_{t, x}}\| v_4 \|_{X^{0, \frac{1}{2}-2\eps}}. 
\end{align*}

\noi
Hence,
by Proposition \ref{PROP:R2},
the contribution to \eqref{nl2} in this case is at most
$\les R^3$
outside a set of probability
\begin{equation}
\leq C\exp\bigg(- c\frac{R^2}{T^\frac{1-2\eps_1}{3}\|\phi\|_{L^2}^2}\bigg)
+
C\exp\bigg(- c\frac{R^2}{T^\frac{1}{2}\|\phi\|_{H^s}^2}\bigg)
\label{ARP1}
\end{equation}

\noi
provided that $s > \frac{\s}{2}$. Note that $s$ needs to be strictly greater than $\frac \s 2$
due to the summations over dyadic blocks.
For the convenience of readers, we briefly show how this follows.
In summing $\|\jb{\nb}^\frac{\s}{2}\P_{N_3}z_3\|_{L^4_{t, x}}$ over dyadic blocks in $N_3$,
we have
\begin{align*}
\sum_{\substack{N_3 \geq 1\\\text{dyadic}}} \|\jb{\nb}^\frac{\s}{2}\P_{N_3}z_3\|_{L^4_{t, x}}
&\leq \Big(\sum_{N_3} N_3^{0-}\Big)^{\frac{3}{4}}
\|\jb{\nb}^{\frac{\s}{2}+}\P_{N_3}z_3\|_{\l^4_{N_3}L^4_{t, x}}\\
& = \Big(\sum_{N_3} N_3^{0-}\Big)^{\frac{3}{4}}
 \|\jb{\nb}^{\frac{\s}{2}+}\P_{N_3}z_3\|_{L^4_{t, x}\l^4_{N_3}}\\
& \leq\Big(\sum_{N_3 } N_3^{0-}\Big)^{\frac{3}{4}}
\|\jb{\nb}^{\frac{\s}{2}+}\P_{N_3}z_3\|_{L^4_{t, x}\l^2_{N_3}}
\les \|\jb{\nb}^{\frac{\s}{2}+}z_3\|_{L^4_{t, x}},
\end{align*}

\noi
where  the last inequality follows from the Littlewood-Paley theory.
By Proposition~\ref{PROP:R2} with $q = r = 4$,
 we obtain the second term in \eqref{ARP1} as long as $ s> \frac \s 2$.
 Moreover, while the terms with $z_1$ and $z_2$ also suffer a slight loss of derivative,
 we can hide the loss in $N_1$ and $N_2$ under the $z_3$-term since $N_3 \geq N_1, N_2$.
Similar comments 
also  apply in the sequel.

\medskip

\noi
{\bf $\bullet$  Subcase (2.b):} $N_3 \sim N_4 \gg N_1, N_2$.

\smallskip

\noi
$\circ$ \underline{Subsubcase (2.b.i):} $N_1, N_2 \ll N_3^\frac{1}{3}$.

We include the detailed calculation only in this case, with similar comments applicable in the following. By Lemmata \ref{LEM:Xsb1} (ii) and  \ref{LEM:Xsb2}, 
with  $ b = \frac 12 +$ and $\dl = 0+$, we have
\begin{align*}
\bigg|\int_{\R\times \R^4}   z_1 z_2& \jb{\nb}^\s  z_3 v_4 dx dt \bigg|
 \les \|z_1 \jb{\nb}^\s z_3\|_{L^2_{t, x}} \|z_2 v_4\|_{L^2_{t, x}}\\
& \les N_1^\frac{3}{2} N_3^{-\frac 12+\s} N_2^{\frac 32 + \eps_1 +\dl} N_4^{-\frac12 + \eps_1}
\prod_{j = 1}^3 \| z_j\|_{X^{0, b}}  \| v_4\|_{X^{0,\frac 12 - 2\eps }}\\
& \les N_1^{\frac{3}{2}-s} N_2^{\frac 32 + \eps_1 -s +\dl}N_3^{-\frac 12+\s - s}  N_4^{-\frac12 + \eps_1}
\prod_{j = 1}^3\| z_j\|_{X^{s, b}} \| v_4\|_{X^{0,\frac 12 - 2\eps }}\\
\intertext{By $N_1, N_2, \ll N_3^\frac{1}{3}$, $N_3 \sim N_4$,
and  Lemma \ref{LEM:Xsb1} (i), we have}
& \ll
T^{0-}N_3^{-\frac{5}{3}s + \s + \frac{4}{3} \eps_1 +\frac 13\dl}
\prod_{j = 1}^3 \|\P_{N_j}\phi^\o\|_{H^s}
 \| v_4\|_{X^{0,\frac 12 - 2\eps }}.
\end{align*}

\noi
Here, we lost a small power of $T$ in applying \eqref{Xsb1}.
Note that  such a  loss in $T$ can be hidden
under $T^\theta$ in \eqref{nl1} and does not cause a problem.
Now, we want the power of the largest frequency  $N_3$ to be strictly negative so that we can sum over dyadic blocks. This requires
\begin{equation}
\frac{5}{3} s >  \s + \frac{4}{3} \eps_1.
\label{nl3}
\end{equation}

\noi
Provided this condition holds, using Lemma \ref{LEM:Hs},  we see that
the contribution to \eqref{nl2} in this case is at most $\les T^{0-} R^3$
outside a set of probability
\begin{equation*}
\leq C\exp\bigg(-c \frac {R^2}{  \|\phi\|_{H^{s}}^{2}}\bigg).
\end{equation*}

\smallskip

\noi
$\circ$ \underline{Subsubcase (2.b.ii):} $N_2\ges N_3^\frac{1}{3} \gg N_1$.

By Lemmata \ref{LEM:Xsb1} and  \ref{LEM:Xsb2},
 we have
\begin{align*}
\bigg|\int_{\R\times \R^4} &  z_1 z_2 \jb{\nb}^\s z_3 v_4 dx dt \bigg|
 \les \|z_2\|_{L^4_{t, x}}
\|\jb{\nb}^\s z_3 \|_{L^4_{t, x}}\|z_1 v_4 \|_{L^2_{t, x}}\\
& \les T^{0-}
N_1^{
\frac{3}{2}  + \eps_1-s+} N_2^{-s}
N_3^{\s-s - \frac 12 +\eps_1}
\|\P_{N_1}\phi^\o\|_{H^s}
\prod_{j = 2}^3 \|\jb{\nb}^s z_j\|_{L^4_{t, x}}
\|v_4\|_{X^{0, \frac 12 - 2\eps}}.
\end{align*}

\noi
Hence,
by Lemma \ref{LEM:Hs} and Proposition \ref{PROP:R2},
the contribution to \eqref{nl2} in this case is at most
$\les T^{0-} R^3$
outside a set of probability
\begin{equation*}
\leq C\exp\bigg(-c \frac {R^2}{  \|\phi\|_{H^{s}}^{2}}\bigg)
+ C\exp\bigg(-c \frac {R^2}{  T^\frac{1}{2} \|\phi\|_{H^{s}}^{2}}\bigg)
\end{equation*}

\noi
provided that \eqref{nl3} is satisfied.

\smallskip

\noi
$\circ$ \underline{Subsubcase (2.b.iii):} $N_1, N_2\geq N_3^\frac{1}{3} $.

By $L^{\frac{9}{2-\eps_1}}L^{\frac{9}{2-\eps_1}}L^{\frac{9}{2-\eps_1}}L^{\frac{3}{1+\eps_1}}$-H\"older's inequality
and Lemmata \ref{LEM:Xsb1} and  \ref{LEM:Xsb2},
  we have
\begin{align*}
\bigg|\int_{\R\times \R^4}  z_1 z_2 \jb{\nb}^\s z_3 v_4 dx dt \bigg|
& \les N_3^{\s - \frac 53 s }
\prod_{j =1}^3 \big\| \jb{\nb}^s z_j\|_{L^\frac{9}{2-\eps_1}_{t, x}}
\|v_4\|_{X^{0, \frac 12 - 2\eps}}.
\end{align*}

\noi
Hence,
by  Proposition \ref{PROP:R2},
the contribution to \eqref{nl2} in this case is at most
$\les R^3$
outside a set of probability
\begin{equation*}
\leq  C\exp\bigg(-c \frac {R^2}{  T^\frac{4 - 2\eps_1 }{9} \|\phi\|_{H^{s}}^{2}}\bigg)
\end{equation*}

\noi
provided that
\begin{equation}
\frac{5}{3} s > \s.
\label{nl4}
\end{equation}

Therefore,
given $s > \frac{3}{5}$,
we  choose $\s = 1+$ and
$\eps = 0+$
 for Case (2)
such that
 \eqref{nl3} and  \eqref{nl4} are satisfied.

\medskip

\noi
{\bf Case (3):} $v v z$ case.

Without loss of generality,   assume $N_1 \geq N_2$.

\medskip

\noi
{\bf $\bullet$  Subcase (3.a):} $N_1 \ges N_3$.

By $L^3L^\frac{6}{1-\eps_1}L^\frac{6}{1-\eps_1}L^\frac{3}{1+\eps_1}$-H\"older's inequality
and
Lemmata \ref{LEM:Xsb1} and  \ref{LEM:Xsb2},
we have
\begin{align*}
\bigg|
\int_{\R\times \R^4} \jb{\nb}^\s  v_1 v_2 z_3 v_4 dx dt \bigg|
& \les  \|v_1\|_{X^{\s, \frac{1}{2}+}}
 \|v_2\|_{X^{1+\eps_1, \frac{1}{2}+}}
 \|z_3\|_{L^\frac{6}{1-\eps_1}_{t, x} }
\|v_4\|_{X^{0, \frac12 - 2\eps}}.
\end{align*}

\noi
Hence,
by  Proposition \ref{PROP:R2},
the contribution to \eqref{nl2} in this case is at most
$
\les R \prod_{j = 1}^2 \|v_j\|_{X^{\s, \frac{1}{2}+}} $
outside a set of probability
\begin{equation}
\leq  C\exp\bigg(-c \frac {R^2}{  T^\frac{1 - \eps_1 }{3} \|\phi\|_{H^{0+}}^{2}}\bigg)
\label{Z1}
\end{equation}

\noi
provided that  $\s > 1 + \eps_1 = 1+ 2\eps+$.
Note that we have
$\|\phi\|_{H^{0+}}$ instead of $\|\phi\|_{L^2}$ in \eqref{Z1}
due to the summation over $N_3$.

\medskip

\noi
{\bf $\bullet$  Subcase (3.b):} $N_3\sim N_4 \gg N_1 \geq N_2$.

By Lemmata \ref{LEM:Xsb1} and  \ref{LEM:Xsb2},
we have
\begin{align*}
\bigg|\int_{\R\times \R^4}  v_1  v_2 & \jb{\nb}^\s z_3  v_4 dx dt \bigg|
 \les \|v_1\|_{L^4_{t, x}} \|\jb{\nb}^\s z_3 \|_{L^4_{t, x}}\|v_2 v_4 \|_{L^2_{t, x}}\\
& \les N_2^{\frac32 + \eps_1 -\s+ } N_3^{\s - s}
N_4^{-\frac 12 +  \eps_1}
 \|v_1\|_{X^{\frac{1}{2}, \frac{1}{2}+}} \|v_2\|_{X^{\s, \frac 12+}} \|\jb{\nb}^{s} z_3\|_{L^4_{t, x}}
 \|v_4\|_{X^{0, \frac 12 - 2\eps}}\\
& \les
N_1^{2 - 2\s  +\eps_1+} N_3^{\s - s - \frac 12 +  \eps_1}
 \|v_1\|_{X^{\s, \frac{1}{2}+}} \|v_2\|_{X^{\s, \frac 12+}} \|\jb{\nb}^{s} z_3\|_{L^4_{t, x}}
  \|v_4\|_{X^{0, \frac 12 - 2\eps}}.
\end{align*}

\noi
Hence,
by  Proposition \ref{PROP:R2},
the contribution to \eqref{nl2} in this case is at most
$
\les R \prod_{j = 1}^2 \|v_j\|_{X^{\s, \frac{1}{2}+}} $
%
outside a set of probability
\begin{equation*}
\leq  C\exp\bigg(-c \frac {R^2}{  T^\frac{1}{2} \|\phi\|_{H^{s}}^{2}}\bigg)
\end{equation*}

\noi
provided that
$2 - 2\s  + \eps_1 < 0$ and
$ s > \s - \frac 12 +  \eps_1$.
Given $s > \frac{1}{2}$,
these conditions are satisfied by taking $\s = 1+$ and $\eps = 0+$.

\medskip

\noi
{\bf Case (4):} $v z z$ case.

Without loss of generality,   assume $N_3 \geq N_2$.

\medskip

\noi
{\bf $\bullet$  Subcase (4.a):} $N_1 \ges N_3$.

By $L^3L^\frac{6}{1-\eps_1}L^\frac{6}{1-\eps_1}L^\frac{3}{1+\eps_1}$-H\"older's inequality
and Lemmata \ref{LEM:Xsb1} and  \ref{LEM:Xsb2},
we have
\begin{align*}
\bigg|\int_{\R\times \R^4} \jb{\nb}^\s v_1 z_2 z_3 v_4 dx dt \bigg|
& \les  \|v_1\|_{X^{\s, \frac{1}{2}+}}
\|z_2\|_{L^\frac{6}{1-\eps_1}_{t, x}} \|z_3\|_{L^\frac{6}{1-\eps_1}_{t, x}} \|v_4\|_{X^{0, \frac 12 - \eps}}.
\end{align*}

\noi
Hence,
by  Proposition \ref{PROP:R2},
the contribution to \eqref{nl2} in this case is at most
$
\les R^2  \|v_1\|_{X^{\s, \frac{1}{2}+}} $
%
outside a set of probability
\begin{equation*}
\leq  C\exp\bigg(-c \frac {R^2}{  T^\frac{1-\eps_1}{3} \|\phi\|_{H^{0+}}^{2}}\bigg).
\end{equation*}

\medskip

\noi
{\bf $\bullet$  Subcase (4.b):} $N_3 \gg N_1$.

First, suppose that $N_2 \sim N_3$.
Then, by Lemmata \ref{LEM:Xsb1} and  \ref{LEM:Xsb2}
(after separating the argument into two cases: $N_1 \leq N_4$ or $N_1 \geq N_4$),
we have
\begin{align*}
\bigg|\int_{\R\times \R^4}  v_1 z_2 & \jb{\nb}^\s z_3 v_4 dx dt \bigg|
 \les \|\jb{\nb}^\frac{\s}{2}z_2\|_{L^4_{t, x}}
\|\jb{\nb}^\frac{\s}{2} z_3 \|_{L^4_{t, x}}\|v_1 v_4 \|_{L^2_{t, x}}\\
& \les
N_1^{1+2\eps_1 - \s+}
N_3^{\s - 2s}
\|v_1\|_{X^{\s, \frac12+}}
\|\jb{\nb}^s z_2\|_{L^4_{t, x}}
\|\jb{\nb}^s z_3 \|_{L^4_{t, x}} \|v_4 \|_{X^{0, \frac 12 -2\eps}}.
\end{align*}

\noi
Hence,
by  Proposition \ref{PROP:R2},
the contribution to \eqref{nl2} in this case is at most
$
\les R^2  \|v_1\|_{X^{\s, \frac{1}{2}+}} $
outside a set of probability
\begin{equation*}
\leq  C\exp\bigg(-c \frac {R^2}{  T^\frac{1}{2} \|\phi\|_{H^{s}}^{2}}\bigg)
\end{equation*}

\noi
provided that $\s >1+2\eps_1 $
and $s > \frac{1}{2}\s$.
Given $s > \frac{1}{2}$,
these conditions are satisfied by taking $\s = 1+$ and $\eps = 0+$.

Hence, it remains to consider the case $N_3 \sim N_4 \gg N_1, N_2$.

\smallskip

\noi
$\circ$ \underline{Subsubcase (4.b.i):} $N_1, N_2\ll N_3^\frac 13$.

By Lemmata \ref{LEM:Xsb1} and  \ref{LEM:Xsb2},
we have
\begin{align*}
\bigg|\int_{\R\times \R^4}  & v_1 z_2  \jb{\nb}^\s z_3 v_4 dx dt \bigg|
 \les \|v_1\jb{\nb}^\s z_3 \|_{L^2_{t, x}}\|z_2 v_4 \|_{L^2_{t, x}}\\
& \les
T^{0-}
N_1^{\frac 32 -\s} N_2^{\frac 32+\eps_1 - s+} N_3^{\s - s-\frac 12 }N_4^{-\frac 12 +  \eps_1}
\|v_1\|_{X^{\s, \frac 12 +}} \prod_{j = 2}^3 \|\P_{N_j}\phi^\o\|_{H^s}\|v_4\|_{X^{0, \frac 12 -2\eps}}\\
& \les
T^{0-}
N_3^{\frac 23 \s - \frac 43 s  + \frac 43\eps_1+}
\|v_1\|_{X^{\s, \frac 12 +}} \prod_{j = 2}^3 \|\P_{N_j}\phi^\o\|_{H^s}\|v_4\|_{X^{0, \frac 12 -2\eps}}.
\end{align*}

\noi
Hence,
by  Lemma \ref{LEM:Hs},
the contribution to \eqref{nl2} in this case is at most
$
\les T^{0-} R^2  \|v_1\|_{X^{\s, \frac{1}{2}+}} $
outside a set of probability
\begin{equation*}
\leq C\exp\bigg(-c \frac {R^2}{  \|\phi\|_{H^{s}}^{2}}\bigg)
\end{equation*}

\noi
provided that
\begin{equation}
s > \frac 12 \s + \eps_1.
\label{nl5}
\end{equation}

\noi
Given $s > \frac{1}{2}$,
this condition is satisfied by taking $\s = 1+$ and $\eps = 0+$.

\smallskip

\noi
$\circ$ \underline{Subsubcase (4.b.ii):} $N_1\ll N_3^{\frac 13} \les N_2$.

By Lemmata \ref{LEM:Xsb1} and  \ref{LEM:Xsb2},
we have
\begin{align*}
\bigg|\int_{\R\times \R^4}  v_1 z_2 & \jb{\nb}^\s z_3 v_4 dx dt \bigg|
 \les \|z_2\|_{L^4_{t, x}} \|\jb{\nb}^\s z_3 \|_{L^4_{t, x}}\|v_1 v_4 \|_{L^2_{t, x}}\\
& \les
N_1^{\frac32 +\eps_1 - \s+ } N_2^{-s} N_3^{\s - s - \frac 12 + \eps_1}
\|v_1\|_{X^{\s, \frac 12+}} \prod_{j = 2}^3 \|\jb{\nb}^s z_j\|_{L^4_{t, x}}\|v_4\|_{X^{0, \frac 12-2\eps}}\\
& \les
N_3^{\frac 23 \s - \frac 43 s  +\frac 43\eps_1+}
\|v_1\|_{X^{\s, \frac 12+}} \prod_{j = 2}^3 \|\jb{\nb}^s z_j\|_{L^4_{t, x}}\|v_4\|_{X^{0, \frac 12-2\eps}}.
\end{align*}

\noi
Hence,
by Proposition \ref{PROP:R2},
the contribution to \eqref{nl2} in this case is at most
$
\les R^2  \|v_1\|_{X^{\s, \frac{1}{2}+}} $
outside a set of probability
\begin{equation*}
\leq C\exp\bigg(-c \frac {R^2}{ T^\frac 12 \|\phi\|_{H^{s}}^{2}}\bigg)
\end{equation*}

\noi
provided that \eqref{nl5} is satisfied.

\smallskip

\noi
$\circ$ \underline{Subsubcase (4.b.iii):} $N_2 \ll N_3^{\frac 13} \les N_1$.

By Lemmata \ref{LEM:Xsb1} and  \ref{LEM:Xsb2},
we have
\begin{align*}
\bigg|\int_{\R\times \R^4}  & v_1 z_2  \jb{\nb}^\s z_3 v_4 dx dt \bigg|
 \les \|v_1\|_{L^3_{t, x}} \|\jb{\nb}^\s z_3 \|_{L^6_{t, x}}\|z_2 v_4 \|_{L^2_{t, x}}\\
& \les
T^{0-} N_1^{ - \s} N_2^{\frac 32 +\eps_1 - s+}
N_3^{\s - s -\frac 12+ \eps_1 }
\|v_1\|_{X^{\s, \frac {1}{2}+}} \|\P_{N_2}\phi^\o \|_{H^s}
\|\jb{\nb}^s z_3\|_{L^6_{t, x}}\|v_4\|_{X^{0, \frac 12-2\eps}}\\
& \les
T^{0-}
N_3^{\frac 23 \s - \frac 43 s  +\frac 43 \eps_1+}
\|v_1\|_{X^{\s, \frac {1}{2}+}} \|\P_{N_2}\phi^\o \|_{H^s}
\|\jb{\nb}^s z_3\|_{L^6_{t, x}}\|v_4\|_{X^{0, \frac 12-2\eps}}.
\end{align*}

\noi
Hence,
by Lemma \ref{LEM:Hs} and Proposition \ref{PROP:R2},
the contribution to \eqref{nl2} in this case is at most
$
\les T^{0-} R^2  \|v_1\|_{X^{\s, \frac{1}{2}+}} $
outside a set of probability
\begin{equation*}
\leq
C\exp\bigg(-c \frac {R^2}{  \|\phi\|_{H^{s}}^{2}}\bigg)
+
C\exp\bigg(-c \frac {R^2}{ T^\frac 13 \|\phi\|_{H^{s}}^{2}}\bigg)
\end{equation*}

\noi
provided that \eqref{nl5} is satisfied.

\smallskip

\noi
$\circ$ \underline{Subsubcase (4.b.iv):} $N_1, N_2\ges N_3^{\frac 13}$.

By $L^3L^\frac{6}{1-\eps_1}L^\frac{6}{1-\eps_1}L^\frac{3}{1+\eps_1}$-H\"older's inequality
and Lemmata \ref{LEM:Xsb1} and  \ref{LEM:Xsb2},
we have
\begin{align*}
\bigg|\int_{\R\times \R^4}  v_1 z_2 \jb{\nb}^\s z_3 v_4 dx dt \bigg|
& \les \|v_1\|_{L^3_{t, x}} \|z_2\|_{L^\frac {6}{1-\eps_1}_{t, x}}
\|\jb{\nb}^\s  z_3 \|_{L^\frac {6}{1-\eps_1}_{t, x}}\| v_4 \|_{L^\frac{3}{1+\eps_1}_{t, x}}\\
& \les
N_1^{- \s} N_2^{-s} N_3^{\s-s}
\|v_1\|_{X^{\s, \frac 12+}} \prod_{j = 2}^3 \|\jb{\nb}^s z_j\|_{L^\frac {6}{1-\eps_1}_{t, x}}\|v_4\|_{X^{0, \frac 12-2\eps}}\\
& \les
 N_3^{\frac 23 \s-\frac 43s}
\|v_1\|_{X^{\s, \frac 12+}} \prod_{j = 2}^3 \|\jb{\nb}^s z_j\|_{L^\frac {6}{1-\eps_1}_{t, x}}\|v_4\|_{X^{0, \frac 12-2\eps}}.
\end{align*}

\noi
Hence,
by  Proposition \ref{PROP:R2},
the contribution to \eqref{nl2} in this case is at most
$
\les R^2  \|v_1\|_{X^{\s, \frac{1}{2}+}} $
outside a set of probability
\begin{equation*}
\leq
C\exp\bigg(-c \frac {R^2}{ T^\frac{1-\eps_1}{3} \|\phi\|_{H^{s}}^{2}}\bigg)
\end{equation*}

\noi
provided that $s > \frac {1}{2}\s$.
Given $s> \frac 12$, this condition is satisfied by setting $\s = 1+$.
\end{proof}

\end{document}